\title{\Large\scshape A twisted homology fibration criterion \\ and the twisted group-completion theorem}
\author{\small Jeremy Miller and Martin Palmer}
\date{\small\today}

\documentclass[reqno,10pt]{article}
\usepackage{etex}

\usepackage[style=alphabetic,backref,backend=bibtex,isbn=false,url=false,maxbibnames=99]{biblatex}
\renewbibmacro{in:}{}
\newbibmacro{string+doi}[1]{\iffieldundef{doi}{\iffieldundef{url}{#1}{\href{\thefield{url}}{#1}}}{\href{http://dx.doi.org/\thefield{doi}}{#1}}}
\DeclareFieldFormat*{title}{\usebibmacro{string+doi}{\emph{#1}}}
\DefineBibliographyStrings{english}{%
  backrefpage = {cited on p.},
  backrefpages = {cited on pp.},
}

\usepackage{amsmath,amscd,amsbsy,amsfonts,amsthm}
\usepackage{latexsym,amssymb,mathrsfs,bbm}
\usepackage[all]{xy}
\usepackage{graphicx}
\usepackage{color}
\usepackage{tikz}
\usetikzlibrary{arrows,calc,positioning,decorations.pathreplacing}
\usepackage[left=4cm,right=4cm]{geometry}
\usepackage{mhsetup}
\usepackage{mathtools}
\usepackage[bottom]{footmisc}
\renewcommand{\footnoterule}{%
  \kern -3pt
  \hrule width \textwidth height 0.4pt
  \kern 2.6pt
}

\usepackage[format=plain,indention=1em,labelsep=quad,labelfont={up},textfont={footnotesize},margin=2em]{caption}

\usepackage[colorlinks,citecolor=blue,linkcolor=blue,urlcolor=blue,filecolor=blue,breaklinks=true]{hyperref}
\usepackage{footnotebackref}

\makeatletter
\renewcommand*{\@seccntformat}[1]{\upshape \csname the#1\endcsname.\hspace{1ex}}
\renewcommand*{\section}{\@startsection{section}{1}{\z@}%
  {2.5ex \@plus 1ex \@minus 0.2ex}%
  {1.5ex \@plus 0.2ex}%
  {\normalfont\large\bfseries}}
\renewcommand*{\subsection}{\@startsection{subsection}{2}{\z@}%
  {2.5ex \@plus 1ex \@minus 0.2ex}%
  {1.5ex \@plus 0.2ex}%
  {\normalfont\normalsize\bfseries}}
\renewcommand*{\subsubsection}{\@startsection{subsubsection}{3}{\z@}%
  {2.5ex \@plus 1ex \@minus 0.2ex}%
  {-1.5ex \@plus -0.2ex}%
  {\normalfont\normalsize\bfseries}}
\renewcommand*{\paragraph}{\@startsection{paragraph}{4}{\z@}%
  {2.5ex \@plus 1ex \@minus 0.2ex}%
  {-1.5ex \@plus -0.2ex}%
  {\normalfont\normalsize\bfseries}}
\renewcommand*{\subparagraph}{\@startsection{subparagraph}{5}{\z@}%
  {2.5ex \@plus 1ex \@minus 0.2ex}%
  {-1.5ex \@plus -0.2ex}%
  {\normalfont\normalsize\slshape}}
\makeatother

\numberwithin{equation}{section}
\numberwithin{figure}{section}

\newcommand{\emb}{\ensuremath{\hookrightarrow}}

\newcommand{\id}{\mathrm{id}}

\newcommand{\hocofib}{\mathit{hocofib}}
\newcommand{\hofib}{\mathit{hofib}}
\newcommand{\hocolim}{\mathit{hocolim}}
\newcommand{\colim}{\mathit{colim}}

\newcommand{\Aut}{\mathrm{Aut}}

\newcommand{\point}{\ensuremath{\text{\itshape pt}}}

\newcommand{\incl}[3][right]%
{%
\draw[<-,>=#1 hook] #2 to ($ #2!0.5!#3 $);
\draw[->] ($ #2!0.5!#3 $) to #3;%
}
\newcommand{\inclusion}[5][right]%
{%
\draw[<-,>=#1 hook] #4 to ($ #4!0.5!#5 $) node[#2,font=\small]{#3};
\draw[->] ($ #4!0.5!#5 $) to #5;%
}

\newcommand{\cA}{\mathcal{A}}
\newcommand{\cB}{\mathcal{B}}

\newcommand{\cF}{\mathcal{F}}
\newcommand{\cG}{\mathcal{G}}
\newcommand{\cH}{\mathcal{H}}

\newcommand{\cM}{\mathcal{M}}

\newcommand{\cU}{\mathcal{U}}

\newcommand{\cW}{\mathcal{W}}

\newcommand{\bN}{\mathbb{N}}

\newcommand{\bR}{\mathbb{R}}

\newcommand{\bZ}{\mathbb{Z}}

\newenvironment{itemize2}%
{\begin{itemize}

\setlength{\itemsep}{1pt}
\setlength{\parskip}{0pt}
\setlength{\parsep}{0pt}}%
{\end{itemize}}

\renewcommand{\cM}{M}

\newcommand{\m}{\to}

\newcommand{\F}{\mathcal{F}}

\newcommand{\U}{\mathcal{B}}

\newcommand{\CC}{\mathfrak{C}}
\newcommand{\TT}{\mathfrak{T}}
\renewcommand{\leq}{\leqslant}
\renewcommand{\geq}{\geqslant}
\newcommand{\snorm}[1]{\lvert #1\rvert}
\newcommand{\dnorm}[1]{\lVert #1\rVert}

\theoremstyle{plain}
\newtheorem{theorem}{Theorem}[section]
\newtheorem{lemma}[theorem]{Lemma}

\newtheorem{proposition}[theorem]{Proposition}
\newtheorem{corollary}[theorem]{Corollary}

\theoremstyle{definition}
\newtheorem{definition}[theorem]{Definition}
\newtheorem{remark}[theorem]{Remark}

\begin{document}
\maketitle
{
\makeatletter
\renewcommand*{\BHFN@OldMakefntext}{}
\makeatother
}

\begin{abstract}
\noindent The purpose of this note is to clarify some details in McDuff and Segal's proof of the group-completion theorem in \cite{MSe} and generalize this and the homology fibration criterion of \cite{Mc1} to homology with twisted coefficients. This will be used in \cite{MP} to identify the limiting homology of ``oriented'' configuration spaces, which doubly cover the classical configuration spaces of distinct unordered points in a manifold.
\end{abstract}


\section{Introduction}

\paragraph{Motivation.}
We say that a sequence of spaces $\{Y_k\}$ indexed by the natural numbers exhibits \emph{homological stability} if the homology groups $H_i(Y_k)$ are independent of $k$ for $k\gg i$. There are many examples of this phenomenon, for example classifying spaces of general linear groups \cite{Q,Charney1980}, mapping class groups of orientable and non-orientable surfaces \cite{Ha,Wahl2008}, automorphism groups of free groups \cite{Hatcher1995,HatcherVogtmann1998,HatcherWahl2010}, moduli spaces of instantons \cite{BHMM}, configuration spaces of unordered particles in an open manifold \cite{Mc1,Se}, and ``oriented'' versions of these configuration spaces \cite{Palmer2013}, which are obtained from the ordered configuration spaces by quotienting by the action of the alternating groups rather than the symmetric groups.

When a sequence of spaces exhibits homological stability, the next natural question is how to compute its stable, or limiting, homology. Typically the sequence of spaces comes equipped with maps $Y_k \to Y_{k+1}$ which induce the homology equivalences in a range, and in many examples one can construct a computationally more tractable space $Z$ and a homology equivalence
$\hocolim_k (Y_k)\to Z$. Within the range of homological stability, the problem of computing the homology of the spaces $Y_k$ is then reduced to computing the homology of $Z$. This has been carried out for example for unordered configuration spaces \cite{Mc1}, spaces of rational functions \cite{Se}, mapping class groups of surfaces \cite{MW,Wahl2008} and automorphism groups of free groups \cite{Galatius2011}.

\paragraph{Homology fibrations and group-completion.}
Two important tools for studying stable homology are the homology fibration criterion of \cite{Mc1} (which is the analogue of a criterion for quasi-fibrations from \cite{DT}) and the group-completion theorem of \cite{MSe}. The group-completion theorem implies that if the union $Y=\bigsqcup Y_k$ assembles to form a homotopy-commutative topological monoid, then the limiting space $Z$ can be taken to be a path-component of $\Omega BY$. Here $\Omega$ denotes the based loop space functor and $B$ denotes the classifying space functor. More generally, the group-completion theorem is useful for identifying the homology of modules over a topological monoid $M$: if $X$ is a module over $M$ and the action is by homology equivalences, then $X$ is homology equivalent to $\hofib(X/\!\!/M \to BM)$, where $X/\!\!/M = X \times_M EM$ is the Borel construction. This recovers the previous example by setting $M=Y$ and $X=\hocolim_k(Y_k)$. It is harder to recognize when the homology fibration criterion applies, but it is more versatile than the group-completion theorem, and is a useful tool whenever one is studying the homology of spaces which are stitched together along homology equivalences. Almost all examples of stable homology calculations that we are aware of use the group-completion theorem, the homology fibration criterion or both.

The goal of this note is to prove versions of the group-completion theorem and the homology fibration criterion which apply to homology with \emph{twisted coefficients} (Theorems \ref{pGroupCompletion} and \ref{strongcrit} respectively). In fact, we work more generally with homology equivalences and homology fibrations relative to a fixed class $\CC$ of twisted coefficient systems. An important example of this is the class $\CC=\mathfrak{Ab}$ of abelian coefficient systems (see \S\ref{ssDef}). In many cases, these results allow one to compute the stable homology of the spaces $Y_k$ with twisted coefficients pulled back from $Z$. In \cite{MP}, we will use this to show that the stable homology of oriented configuration spaces is given by the homology of a certain double cover of a section space. We refer to \cite{Palmer2013} for a discussion of oriented configuration spaces and to \cite{Mc1} for a discussion of the relevant section space.

\paragraph{Context.}
In Remark 2 of \cite{MSe}, McDuff and Segal stated a version of the group-completion theorem for twisted coefficient systems. This was again stated in Lemma 3.1 of \cite{Mc3}, but no details were given there or in \cite{MSe} about generalizing the proof to this situation. Another claim implicit in Remark 2 of \cite{MSe} is that the natural action of a homotopy-commutative topological monoid $M$ on $M_\infty = \hocofib(M\to M\to \cdots)$ is by abelian homology equivalences, where the homotopy cofiber is taken over a sequence of maps $-\cdot m_k\colon M\to M$, with the sequence $\{m_k\}$ in $M$ chosen so that $H_*(M_\infty) = H_*(M)[\pi_0(M)^{-1}]$. It is not hard to see that the action is by ordinary homology equivalences, but the statement for all abelian coefficient systems is much more subtle. However, a detailed proof of this statement was given recently in \cite{Randal-Williams2013Groupcompletion}. This note is complementary to \cite{Randal-Williams2013Groupcompletion} in that we prove the claimed version of the group-completion theorem for abelian coefficient systems, which therefore (by \cite{Randal-Williams2013Groupcompletion}) applies to the above example of $M$ acting on $M_\infty$.

Since the publication of \cite{MSe} some mistakes in their proof of the group-completion theorem have been discovered, and methods for addressing these gaps are discussed by G.\ Segal and D.\ McDuff respectively in \cite[pp.\ 56--57]{Se} and \cite[pp.\ 109--110]{Mc3}. In this note we give a complete account of the proof of the group-completion theorem, using some techniques of \cite{GRW14} but closely following the ideas of \cite{MSe}.

\paragraph{History.}
The version of the group-completion theorem proved in \cite{MSe} originally grew out of a theorem of Barratt and Priddy \cite{BarrattPriddy1972}, which was also proved by Quillen (in a preprint which remained unpublished until its inclusion as an appendix of \cite{FriedlanderMazur1994}) and by May \cite{May1975}. Some other proofs of the group-completion theorem have been published since \cite{MSe}, including \cite{Jardine1989} and \cite{Moerdijk1989}, which work with bisimplicial sets. The proof of \cite{Jardine1989} was later generalized in \cite{Tillmann1997} to prove a multiple-object version for topological categories, and this was reproved in \cite{PS} via a more topological treatment in the spirit of \cite{MSe}.

\subsection{Outline}

We work throughout with a class $\CC$ of twisted coefficient systems. We begin in \S\ref{s-prelim} by describing two definitions of $\CC$-homology fibration and give sufficient conditions for these to coincide. In \S\ref{ssCriterion} we then generalize McDuff's homology fibration criterion to a criterion for $\CC$-homology fibrations. Finally, in \S\ref{ss-group-completion} we apply this and the equivalence of the two notions of $\CC$-homology fibration to give a proof of the group completion theorem for twisted coefficient systems.

\subsection{Acknowledgments}

We would like to thank Oscar Randal-Williams and Ulrike Tillmann for several enlightening discussions, and Johannes Ebert for his detailed question \cite{Ebert} on MathOverflow which was likewise enlightening. This note was originally part of the article \cite{MP13}, which was subsequently split in half for length reasons. We would like to thank an anonymous referee for many helpful suggestions and corrections for that article, which have also significantly improved the present note.

\section{Two definitions of homology fibration}\label{s-prelim}

In this section we give two definitions of twisted homology fibration and prove that they are equivalent under reasonable point-set-topological hypotheses. These two notions are twisted analogues of the two definitions of homology fibration introduced in \cite{Mc1} and \cite{MSe} respectively. The equivalence of the two definitions will be important for the proofs of both the twisted homology fibration criterion and the twisted group completion theorem in the later sections.

\subsection{Definitions}\label{ssDef}

All definitions of homology equivalence and homology fibration will depend on a fixed class of local coefficient systems. One can define a local coefficient system on a space $Y$ as either a functor $\pi(Y)\to\mathsf{Ab}$ from the fundamental groupoid of $Y$ to the category of abelian groups or as a bundle of abelian groups over $Y$. It is called \emph{abelian} if, in the bundle viewpoint, the monodromy of any fiber around a commutator loop is trivial. In the functor viewpoint this says that for each object $y\in Y$ the homomorphism $\pi_1(Y,y)=\pi(Y)(y,y)\to \Aut_{\mathsf{Ab}}(\cF(y))$ factors through an abelian group.

\paragraph{A collection of coefficient systems.} Let $\TT$ denote the following category. Its objects are pairs $(X,\F)$ with $X$ a topological space and $\F$ a bundle of abelian groups on $X$. A morphism from $(X_1,\F_1)$ to $(X_2,\F_2)$ in $\TT$ is a continuous map $f\colon X_1 \m X_2$ together with a bundle map $\hat{f}\colon \F_1 \m \F_2$ covering it (i.e., $\hat{f}$ restricts to an isomorphism of abelian groups on each fiber). Throughout this section we fix a subcategory $\CC \subseteq \TT$ with the property that if $g$ is a morphism of $\TT$ whose target is in $\CC$, then $g$ is in $\CC$. We think of $\CC$ as a collection of allowable local coefficient systems and the previous condition can be rephrased as the requirement that this collection be closed under pullbacks. We will be especially interested in the case where $\CC$ is the subcategory $\mathfrak{Ab}$ consisting of all abelian local coefficient systems. However, taking $\CC$ to be all trivial local coefficient systems (the full subcategory on the objects $(X,\F)$ where $\F\to X$ admits a trivialization) or all local coefficient systems ($\CC=\TT$) will also be interesting.

\begin{definition} We call a map $f\colon A \m B$ a \emph{$\CC$-homology equivalence} if the induced map $f_* \colon H_*(A;f^* \F) \m H_*(B;\F)$ is an isomorphism for all local coefficient systems $\F$ with $(B,\F) \in \CC$. When $\CC$ is the subcategory of trivial local coefficient systems, we simply call $f$ a \emph{homology equivalence}. When $\CC=\TT$ we call $f$ an \emph{acyclic map} or a \emph{twisted homology equivalence} and when $\CC=\mathfrak{Ab}$ we call $f$ an \emph{abelian homology equivalence}.
\end{definition}

\label{ss-strong-homology-fib}

We now give two definitions of the property of being a $\CC$-homology fibration. One variant is called a Serre $\CC$-homology fibration since such maps naturally have an associated Serre spectral sequence, and the other variant is called a Leray $\CC$-homology fibration since these maps naturally have an associated Leray spectral sequence.

We denote the homotopy fiber of a map $r\colon Y \m X$ over a point $x\in X$ by $\hofib_x(r)$. More generally, for a subset $U \subseteq X$, the symbol $\hofib_U(r)$ will denote the homotopy fiber product of $U$ and $Y$ over $X$. Concretely, this means the pullback of the diagram%
\begin{center}
\begin{tikzpicture}
[x=1mm,y=1mm]
\node (tr) at (15,12) {$Y$};
\node (bl) at (0,0) {$P_U X$};
\node (br) at (15,0) {$X$};
\draw[->] (bl) to (br);
\draw[->] (tr) to node[right,font=\small]{$r$} (br);
\end{tikzpicture}
\end{center}
where $P_U X = \mathrm{Map}([0,1],\{0\};X,U)$ with the compact-open topology and the horizontal map is evaluation at $1$. Note that this space is sometimes denoted by $U \times_X^h Y$.

\begin{definition}
Let $Z$ be a subspace of a space $X$. A map $r\colon Y \m X$ is called a \emph{Serre $\CC$-homology fibration on $Z$} if for all points $z \in Z$ the natural inclusion $r^{-1}(z) \m  \hofib_z(r|_Z)$ induces an isomorphism on homology for any local coefficient system in $\CC$ pulled back from $\hofib_z(r)$. That is, if $\F$ is a local coefficient system on $\hofib_z(r)$ with $(\hofib_z(r),\F) \in \CC$ and $i\colon r^{-1}(z) \m \hofib_z(r|_Z)$ and $j\colon \hofib_z(r|_Z) \m \hofib_z(r)$ are the natural inclusions, then $i$ induces an isomorphism:
\[
i_*\colon H_*(r^{-1}(z);i^*j^*\F) \m H_*(\hofib_z(r|_Z);j^* \F).
\]
The map $r\colon Y  \m X$ is called simply a Serre $\CC$-homology fibration if it is a Serre $\CC$-homology fibration on $X$.
\end{definition}

When $\CC$ is the subcategory of trivial, abelian or all local coefficients systems we call these maps \emph{Serre homology fibrations}, \emph{abelian Serre homology fibrations} or \emph{twisted Serre homology fibrations} respectively.

\begin{definition}\label{dLerayHF}
Let $X$ be a space and let $Z \subseteq X$ be a locally contractible subspace. A map $r\colon Y \m X$ is called a \emph{Leray $\CC$-homology fibration on $Z$} if there is a basis $\U$ for the topology of $Z$ consisting of contractible sets, such that for all $z\in U\in \U$ and any local coefficient system $\F$ on $\hofib_U(r)$ in $\CC$, the inclusion $i\colon r^{-1}(z) \m r^{-1}(U)$ induces an isomorphism on homology with coefficients pulled back from $\F$. That is, if $j\colon r^{-1}(U) \m \hofib_U(r)$ is the natural inclusion, then
\[
i_* \colon H_*(r^{-1}(z);i^* j^* \F) \m H_*(r^{-1}(U);j^* \F)
\]
is an isomorphism. The map $r\colon Y \m X$ is called simply a Leray $\CC$-homology fibration if it is a Leray $\CC$-homology fibration on $X$. We say that such a basis $\U$ \emph{witnesses} that the map $r$ is a Leray $\CC$-homology fibration.
\end{definition}

We likewise define the terms \emph{Leray homology fibration}, \emph{abelian Leray homology fibration} and \emph{twisted Leray homology fibration}.

\subsection{Equivalence of two definitions of homology fibration}

Next we prove that these two definitions of $\CC$-homology fibration are equivalent under reasonable point-set-topological hypotheses. Much of this is implicit but not explicit in the work of \mbox{McDuff} and Segal in \cite{MSe} (Proposition 5, Proposition 6 and Remark 2). The equivalence of these two definitions will be used in Section \ref{ssCriterion} to generalize McDuff's homology fibration criterion and in Section \ref{ss-group-completion} to prove the twisted group completion theorem. To state when the two notions coincide we need the following definition.

\begin{definition}\label{dLocallyStalklike}
We say that a map $r\colon Y\to X$ is \emph{locally stalk-like} over $Z\subseteq X$ if there is a basis $\U$ for the topology of $Z$ such that each $U\in \U$ is contractible and contains a point $z_U$ such that the inclusion $r^{-1}(z_U) \hookrightarrow r^{-1}(U)$ is a weak equivalence. More generally, we say that $r$ is \emph{locally stalk-like on $\CC$-homology} if in the above the inclusion $r^{-1}(z_U) \hookrightarrow r^{-1}(U)$ is an isomorphism on homology for local coefficient systems in $\CC$ which are pulled back from $\hofib_U(r)$.
\end{definition}

Note that by definition any Leray $\CC$-homology fibration is locally stalk-like on $\CC$-homology. This is therefore a necessary condition for a Serre $\CC$-homology fibration to be a Leray $\CC$-homology fibration; by the following, it is also sufficient.

\begin{proposition}\label{SimpL}
Let $r\colon Y\to X$ be a map of spaces which is locally stalk-like on $\CC$-homology over a subspace $Z\subseteq X$. If $r$ is a Serre $\CC$-homology fibration on $Z$, then it is also a Leray $\CC$-homology fibration on $Z$.
\end{proposition} 
\begin{proof}
Let $\U$ be a basis for the topology of $Z$ as in Definition \ref{dLocallyStalklike}. Let $z \in U \in \U$ be arbitrary and let $z_U$ be as above. Consider the following commutative diagram:
\begin{center}
\begin{tikzpicture}
[x=1.2mm,y=1mm]
\node (tl) at (0,12) {$r^{-1}(z)$};
\node (tm) at (25,12) {$r^{-1}(U)$};
\node (tr) at (50,12) {$r^{-1}(z_U)$};
\node (bl) at (0,0) {$\hofib_{z}(r|_Z)$};
\node (bm) at (25,0) {$\hofib_U(r|_Z)$};
\node (br) at (50,0) {$\hofib_{z_U}(r|_Z).$};
\draw[->] (tl) to (tm);
\draw[->] (tr) to (tm);
\draw[->] (bl) to (bm);
\draw[->] (br) to (bm);
\draw[->] (tl) to (bl);
\draw[->] (tm) to (bm);
\draw[->] (tr) to (br);
\end{tikzpicture}
\end{center}
Since $\{z\} \m U$ and $\{z_U\} \m U$ are homotopy equivalences, so are $\hofib_z(r|_Z) \m \hofib_U(r|_Z)$ and $\hofib_{z_U}(r|_Z) \m \hofib_U(r|_Z) $. Fix a local coefficient system $\cF$ in $\CC$ on $\hofib_U(r)$. Then by assumption, $r^{-1}(z_U) \to r^{-1}(U)$ is an isomorphism on homology with coefficients pulled back from $\cF$. Since $r$ is a Serre $\CC$-homology fibration on $Z$, the maps $r^{-1}(z) \m \hofib_z(r|_Z)$ and $r^{-1}(z_U) \m \hofib_{z_U}(r|_Z)$ also induce isomorphisms on homology with coefficients pulled back from $\cF$. Thus the same is true for the map $r^{-1}(z) \m r^{-1}(U)$ and hence $r$ is a Leray $\CC$-homology fibration on $Z$, witnessed by the basis $\U$.
\end{proof}

The statement of the converse is simpler -- every Leray $\CC$-homology fibration over a Hausdorff base space is a Serre $\CC$-homology fibration -- but the proof is more involved. We first recall the following result of Galatius and Randal-Williams. 

\begin{proposition}[{\cite[Corollary 2.9]{GRW14}}]\label{GRW}
Let $C$ be a Hausdorff space and $B_\bullet$ a semi-simplicial set. Let $A_\bullet \subseteq B_\bullet \times C$ be a sub-semi-simplicial space which in each simplicial degree is an open subset. For $c \in C$ define a sub-semi-simplicial set $A_\bullet(c) \subseteq B_\bullet$ by the property $A_\bullet \cap (B_\bullet \times \{c\}) = A_\bullet(c) \times \{c\}$. If its geometric realization $\lVert  A_\bullet(c) \rVert$ is $n$-connected for each $c \in C$ then the map $\lVert  A_\bullet \rVert  \m C$ is $(n+1)$-connected. 
\end{proposition}

We also recall a spectral sequence computing the twisted homology of the geometric realization of a semi-simplicial space. 

\begin{lemma}\label{lSpectralSequence}
Let $A_\bullet$ be a semi-simplicial object in $\TT$ and let $X_\bullet$ be its underlying semi-simplicial space. Denote the local coefficient system on $X_n$ by $\cA_n$. There is an induced local coefficient system on $\lVert X_\bullet \rVert$, which we denote by $\cA$. Then there is a first quadrant, homologically graded spectral sequence
\[
E^1_{p,q} \cong H_q(X_p;\cA_p) \qquad\Rightarrow\qquad H_{p+q}(\lVert X_\bullet \rVert ; \cA),
\]
and the edge homomorphism $H_q(X_0;\cA_0) \cong E^1_{0,q} \twoheadrightarrow E^{\infty}_{0,q} \hookrightarrow H_q(\lVert X_\bullet \rVert ; \cA)$ is the map on $\cA$-twisted homology induced by the inclusion of the vertices $X_0 \to \lVert X_\bullet \rVert$.
\end{lemma}

Note that if $X_\bullet$ has an augmentation $X_0 \to X_{-1}$ then any local coefficient system $\cA_{-1}$ on $X_{-1}$ gives a semi-simplicial object in $\TT$ over $X_\bullet$ by pulling back along the compositions of face maps $X_n\to X_{-1}$.

\begin{proof}
There is a local coefficient system on $\bigsqcup_{i\geq 0} (\Delta^i \times X_i)$ pulled back from the coefficient systems $\cA_i$ on $X_i$. Since the face maps of $X_\bullet$ are covered by bundle maps, this descends to a well-defined local coefficient system $\cA$ on the geometric realization $\lVert X_\bullet \rVert$. Note that the pullback of $\cA$ along the characteristic map $\Delta^i \times X_i \to \dnorm{X_\bullet}$ is isomorphic to the pullback of $\cA_i$ along the projection $\Delta^i \times X_i \to X_i$.

The construction of the spectral sequence with local coefficient systems is essentially the same as the construction with trivial coefficients. Filtering $\lVert X_\bullet \rVert$ by its skeleta gives a first quadrant, homologically graded spectral sequence converging to $H_*(\lVert X_\bullet \rVert; \cA)$ with $E^1_{p,q}$ isomorphic to $H_{p+q}(\lVert X_\bullet \rVert^{p}, \lVert X_\bullet \rVert^{p-1};\cA)$, where $\cA$ also denotes its restriction to skeleta. The edge homomorphism on the $q$-axis is the map on homology induced by the inclusion of the zero-space of the filtration, i.e.\ the vertices $X_0 = \lVert X_\bullet \rVert^{0}$. By excision and the observation at the end of the previous paragraph we have
\begin{align*}
E^1_{p,q} \cong H_{p+q}(\lVert X_\bullet \rVert^{p}, \lVert X_\bullet \rVert^{p-1};\cA) &\cong H_{p+q}(\Delta^p \times X_p, \partial \Delta^p \times X_p ; \cA_p) \\
&\cong H_q(X_p;\cA_p).\qedhere
\end{align*}
\end{proof}

Before we prove that Leray $\CC$-homology fibrations are Serre $\CC$-homology fibrations, we prove the following lemma which generalizes Proposition 6 of \cite{MSe}.

\begin{lemma}\label{contractibleCase}
Let $r\colon Y \m X$  be a map of Hausdorff spaces with $X$ weakly contractible and let $\U$ be a basis for the topology of $X$ consisting of contractible sets. Fix a local coefficient system $\F \m Y$ and assume that for all $x \in U \in \U$ the inclusion $r^{-1}(x) \m r^{-1}(U)$ is a homology equivalence with coefficients pulled back from $\F$. Then for any $x\in X$ the inclusion $i\colon r^{-1}(x) \m Y$ is also a homology equivalence with with coefficients pulled back from $\F$.
\end{lemma} 

\begin{proof}
The idea of the proof is to use semi-simplicial spaces to build a Leray-type spectral sequence computing $H_*(Y;\F)$, whose $E_2$ page is the homology of the base $X$ with coefficients in the homology of the fibers of $r$ (with twisted coefficients). Since $X$ is contractible, we will be able to show that the spectral sequence collapses at the $E_2$ page and then that $i_*\colon H_*(r^{-1}(x);i^*\F) \m H_*(Y;\F)$ is an isomorphism. 

\paragraph{\normalfont\itshape Some semi-simplicial spaces.}
Let $X_\bullet$ be the semi-simplicial nerve of the topological poset $\bigsqcup_{\{U \in\cB\}} U$ in which $(U,y)\leq (V,x)$ iff $y=x$ and $U\subseteq V$. Similarly, let $Y_\bullet$ be the semi-simplicial nerve of the topological poset $\bigsqcup_{\{U\in\cB\}} r^{-1}(U)$ in which $(U,y)\leq (V,x)$ iff $y=x$ and $r^{-1}(U) \subseteq r^{-1}(V)$. So the spaces of $n$-simplices are:
\begin{align*}
X_n &= \textstyle{\bigsqcup}_{\{U_0 \subseteq \dotsb \subseteq U_n\}} U_0 &
Y_n &= \textstyle{\bigsqcup}_{\{r^{-1}(U_0) \subseteq \dotsb \subseteq r^{-1}(U_n)\}} r^{-1}(U_0).
\end{align*}

There is an augmentation map $\xi \colon X_0 = \bigsqcup_{\{U\in\cB\}}U \to X$ given by the sum of the inclusion maps, and similarly an augmentation map $\phi \colon Y_0 = \bigsqcup_{\{U\in\cB\}}r^{-1}(U) \to Y$. To show that the augmentation map $\xi$ is a weak equivalence, we apply Proposition \ref{GRW} with $C=X$, $A_\bullet=X_\bullet$ and $B_\bullet$ the semi-simplicial nerve of the poset $\U$ partially ordered by inclusion. For each $i$, the space $X_i$ naturally sits as a subspace of $X \times B_i$ and this subspace is open because the elements of $\U$ are open sets. The semi-simplicial set $X_\bullet(x)$ is the nerve of the subposet of $\U$ consisting of all basic open sets which contain $x$. Since $\U$ is a basis this is filtered, i.e.\ a directed set, so $\dnorm{X_\bullet(x)}$ is contractible. Proposition \ref{GRW} therefore implies that $\lVert \xi \rVert \colon \lVert X_\bullet \rVert \m X$ is a weak equivalence. A similar argument shows that $\lVert \phi \rVert \colon \lVert Y_\bullet \rVert \m Y$ is also a weak equivalence.

\paragraph{\normalfont\itshape Spectral sequences.}
We now use these semi-simplicial spaces to construct spectral sequences. The coefficient system $\F$ on $Y=Y_{-1}$ determines a semi-simplicial object $\cF_\bullet$ in $\TT$ over $Y_\bullet$ (c.f.\ the remark after Lemma \ref{lSpectralSequence}). By Lemma \ref{lSpectralSequence} and the fact that $\dnorm{\xi}$ is a weak equivalence we have a spectral sequence $E_{**}^*$ with $E_{p,q}^1 \cong H_q(Y_p ; \cF_p)$ converging to $H_*(Y;\cF)$.

We next define for each $k\geq 0$ a semi-simplicial object $\cH_{k,\bullet}$ in $\TT$ whose underlying semi-simplicial space is $X_\bullet$. We define the total space of $\cH_{k,n}$ to be
\[
\bigsqcup_{U_0 \subseteq \dotsb \subseteq U_n} U_0 \times H_k(r^{-1}(U_0);j_0^*\F),
\]
where $j_i \colon r^{-1}(U_i) \m Y$ is the inclusion map. It is given the structure of a bundle of abelian groups over $X_n$ by projecting onto the first factor; over each $U_0$ corresponding to a nested sequence $U_0 \subseteq \dotsb \subseteq U_n$ it is a trivial bundle. For $i>0$, the face map $d_i\colon X_{n+1} \m X_n$ is just a sum of identity maps $U_0 \m U_0$. We likewise define the $i$th face map $\cH_{k,n+1} \m \cH_{k,n}$ for $i>0$ to be the sum of identity maps $U_0 \times H_k(r^{-1}(U_0);j_0^*\F) \m U_0 \times H_k(r^{-1}(U_0);j_0^*\F)$. This is clearly a bundle map. The face maps $d_0\colon X_{n+1} \m X_{n}$ are given by sums of inclusions $\iota\colon U_0 \emb U_1$. We cover this with the corresponding sum of the maps
\[
\iota \times \tilde{\iota}_*\colon U_0 \times H_k(r^{-1}(U_0);j_0^*\F) \longrightarrow U_1 \times H_k(r^{-1}(U_1);j_1^*\F),
\]
where $\tilde{\iota}$ is the inclusion $r^{-1}(U_0) \emb r^{-1}(U_1)$. To see that $\iota \times \tilde{\iota}_*$ is a bundle map we need to check that $\tilde{\iota}_*\colon  H_k(r^{-1}(U_0);j_0^*\F) \m  H_k(r^{-1}(U_1);j_1^*\F)$ is an isomorphism. To see this, choose a point $u\in U_0$ and consider the commutative diagram
\begin{center}
\begin{tikzpicture}
[x=1mm,y=1mm]
\node (l) at (0,6) {$r^{-1}(u)$};
\node (t) at (30,12) {$r^{-1}(U_0)$};
\node (b) at (30,0) {$r^{-1}(U_1)$};
\node (r) at (60,6) {$Y$};
\incl{(l)}{(t)}
\incl{(l)}{(b)}
\inclusion{right}{$\tilde{\iota}$}{(t)}{(b)}
\draw[->] (t) to node[above,font=\small]{$j_0$} (r);
\draw[->] (b) to node[below,font=\small]{$j_1$} (r);
\end{tikzpicture}
\end{center}
\noindent in which each space is given coefficients pulled back from $\cF\to Y$. By hypothesis, the leftmost two inclusions are isomorphisms on homology with these coefficients, and therefore so is $\tilde{\iota}$.\footnote{Note that it would \emph{not} suffice to know that each $U\in\U$ contains some $z_U$ so that $r^{-1}(z_U) \to r^{-1}(U)$ is an isomorphism on homology with coefficients pulled back from $\cF$ (c.f.\ Definition \ref{dLocallyStalklike}), since there is no reason why $z_{U_1}$ should be contained in $U_0$.} This shows that the map $\iota \times \tilde{\iota}_*$ is a bundle map and hence is a morphism in the category $\TT$. Checking that the semi-simplicial identities are satisfied is straightforward. We therefore have a semi-simplicial object $\cH_{k,\bullet}$ in $\TT$ over $X_\bullet$. Denote the induced local coefficient system on $\dnorm{X_\bullet}$ by $\cH_k$. Then by Lemma \ref{lSpectralSequence} there is a spectral sequence ${}_{(k)}E^*_{**}$ with ${}_{(k)}E^1_{p,q} \cong H_q(X_p ; \cH_{k,p})$ converging to $H_*(\lVert X_\bullet \rVert ; \cH_k)$.

\paragraph{\normalfont\itshape Comparing the spectral sequences.}
Note that
\[
{}_{(k)}E^1_{p,q} \cong\, \textstyle{\bigoplus} H_q(U_0 ; H_k(r^{-1}(U_0) ;j_0^*\cF)) = \begin{cases} {\bigoplus} H_k(r^{-1}(U_0) ;j_0^*\cF) & q=0 \\
0 & q>0
\end{cases}
\]
and
\[
E^1_{p,q} \cong\, \textstyle{\bigoplus} H_q(r^{-1}(U_0) ;j_0^*\cF),
\]
where the sum is always over nested subsets $U_0 \subseteq \dotsb \subseteq U_p$ in $\cB$. The differentials also agree, so the chain complex ${}_{(q)}E^1_{*,0}$ is isomorphic to the chain complex $E^1_{*,q}$. But ${}_{(q)}E^*_{**}$ collapses after its second page since each $U_0$ is contractible and so
\begin{align*}
E^2_{p,q} &\cong {}_{(q)}E^2_{p,0} \\
&\cong \text{degree-$p$ part of the limit of ${}_{(q)}E^*_{**}$} \\
&= H_p(\lVert X_\bullet \rVert ; \cH_q).
\end{align*}
But $\dnorm{X_\bullet} \simeq_{\mathrm{w}} X$ is weakly contractible, so $E^2_{p,q} = 0$ for $p>0$ and $E^2_{0,q} \cong H_q(r^{-1}(U); j^* \cF)$ for any $U\in\cB$ with $j\colon r^{-1}(U) \m Y$ the inclusion. Hence $E^*_{**}$ also collapses after its second page, and
\begin{equation}\label{eEdgeHomomorphism}
H_q(r^{-1}(U); j^* \cF) \cong E^2_{0,q} \cong E^{\infty}_{0,q} \cong H_q(Y; \cF).
\end{equation}
Moreover (c.f.\ the identification of the edge homomorphisms in Lemma \ref{lSpectralSequence}) this isomorphism is the map on homology with coefficients in $\cF$ induced by $j$.

Now pick $x\in X$; we would like $i\colon r^{-1}(x) \emb Y$ to induce isomorphisms on homology with coefficients in $\cF$. Choose $U\in\U$ containing $x$, so $i$ factors as the inclusions $r^{-1}(x)\emb r^{-1}(U)\emb Y$. These each induce isomorphisms on homology with coefficients in $\cF$, the first by hypothesis and the second by the paragraph above.
\end{proof}

\begin{proposition}\label{LimpS}
Let $r\colon Y \m X$ be a map of topological spaces and let $Z \subseteq X$ be a Hausdorff subspace such that $r^{-1}(Z)$ is also Hausdorff. If $r$ is a Leray $\CC$-homology fibration on $Z$, then it is also a Serre $\CC$-homology fibration on $Z$.
\end{proposition} 

\begin{proof}
Let $\cB$ be a basis for the topology of $Z$ witnessing that $r$ is a Leray $\CC$-homology fibration on $Z$. Fix a point $z_0\in Z$ and a local coefficient system $\cF\to \hofib_{z_0}(r)$ in $\CC$. We need to show that the maps
\[
r^{-1}(z_0) \xrightarrow{i} \hofib_{z_0}(r|_Z) \xrightarrow{j} \hofib_{z_0}(r)
\]
induce an isomorphism $H_*(r^{-1}(z_0) ; i^* j^* \cF) \to H_*(\hofib_{z_0}(r|_Z) ; j^* \cF)$ --- in other words that $i$ is a $\cF$-homology isomorphism. As in Proposition 5 of \cite{MSe}, our goal is to reduce this to the case where the base space is contractible so we can use Lemma \ref{contractibleCase}.

Let $PZ$ be the pathspace $\mathrm{Map}([0,1],\{0\};Z,z_0)$ and let $p\colon PZ \to Z$ be the projection given by evaluating at $1$. Let $\alpha_0 \in PZ$ be the constant path at $z_0$. We have a pullback square:
\begin{equation}\label{ePullbackSquare}
\centering
\begin{split}
\begin{tikzpicture}
[x=1mm,y=1mm]
\node (tl) at (0,15) {$F = \hofib_{z_0}(r|_Z)$};
\node (tr) at (30,15) {$PZ$};
\node (bl) at (0,0) {$r^{-1}(Z)$};
\node (br) at (30,0) {$Z.$};
\draw[->] (tl) to node[above,font=\small]{$g$} (tr);
\draw[->] (bl) to node[below,font=\small]{$r|_Z$} (br);
\draw[->] (tl) to node[left,font=\small]{$q$} (bl);
\draw[->] (tr) to node[right,font=\small]{$p$} (br);
\node at (5,10) {$\lrcorner$};
\end{tikzpicture}
\end{split}
\end{equation}
Also we can canonically identify $g^{-1}(\alpha_0) = r^{-1}(z_0)$, and under this identification the inclusion $g^{-1}(\alpha_0) \hookrightarrow F$ is the map $i\colon r^{-1}(z_0) \to \hofib_{z_0}(r|_Z)$.

Define an open covering $\cW$ of $PZ$ as follows. A sequence $c=(U_1,V_1,\dotsc,U_n,V_n)$ of elements of $\cB$ is called a \emph{chain} if we have:
\[
U_1 \supseteq V_1 \subseteq U_2 \supseteq \dotsb \subseteq U_n \supseteq V_n.
\]
Given such a chain in $\cB$ and a sequence of numbers $t=(t_0,\dotsc,t_n)$ satisfying $0=t_0 < t_1 < \dotsb < t_n = 1$ we define $W_{c,t} \subseteq PZ$ to be the subset of paths $\alpha$ with $\alpha(t_i) \in V_i$ and $\alpha([t_{i-1},t_i]) \subseteq U_i$ for $i=1,\dotsc,n$. The covering $\cW$ of $PZ$ is defined to be the collection of all such sets $W_{c,t}$.

We claim that $\cW$ is a basis for $PZ$ witnessing that $g$ is a Leray $\CC$-homology fibration. First, note that $\cW$ forms a basis for the (compact-open) topology of $PZ$.\footnote{It is easy to see that $\cW$ is a basis for \emph{a} topology; one then needs to show that for any path $\gamma\in PZ$ taking a compact subset $K\subseteq [0,1]$ to $B\in\cB$, there is a set of the form $W_{c,t}$ containing $\gamma$ such that $\gamma^\prime(K)\subseteq B$ for all $\gamma^\prime \in W_{c,t}$. This can be done using the fact that for any open neighborhood $U$ of $K$, there is a finite union of closed intervals $K^\prime$ such that $K\subseteq K^\prime \subseteq U$.} Second, the subsets $W_{c,t}$ are contractible since the product of evaluation maps $e\colon W_{c,t} \m V_1 \times \dotsb \times V_n$ is a Hurewicz fibration with contractible base and fibers. In more detail, let $h\colon V_1\times \dotsb \times V_n \times [0,1] \to V_1\times \dotsb \times V_n$ be a homotopy from the identity to the constant map at the point $\vec{v}$. The composition $h\circ (e\times\id_{[0,1]})$ can be lifted up the Hurewicz fibration $e$ to a homotopy $\tilde{h}$ from the identity on $W_{c,t}$ to a map which factors as $W_{c,t} \to e^{-1}(\vec{v}) \hookrightarrow W_{c,t}$. Inserting a contraction for the fiber $e^{-1}(\vec{v})$ into the middle of this composition gives a further homotopy to a constant map.

Now let $\gamma \in W\in \cW$ and choose a local coefficient system $\cG\to\hofib_W(g)$ in $\CC$. We need to show that $g^{-1}(\gamma) \hookrightarrow g^{-1}(W)$ is an isomorphism on homology with $\cG$-coefficients. Consider the diagram
\begin{equation}\label{eComparingFibers}
\centering
\begin{split}
\begin{tikzpicture}
[x=1mm,y=1mm]
\node (t1) at (0,10) {$g^{-1}(\gamma)$};
\node (t2) at (25,10) {$g^{-1}(W)$};
\node (t3) at (50,10) {$\hofib_W(g)$};
\node (b1) at (0,0) {$r^{-1}(\gamma(1))$};
\node (b2) at (25,0) {$r^{-1}(V_n)$};
\node (b3) at (50,0) {$\hofib_{V_n}(r|_Z)$};
\incl{(t1)}{(t2)}
\incl{(b1)}{(b2)}
\draw[->] (t2) to (t3);
\draw[->] (b2) to (b3);
\draw[->] (t2) to (b2);
\draw[->] (t3) to (b3);
\node at (0,5) {\rotatebox{90}{$=$}};
\end{tikzpicture}
\end{split}
\end{equation}
where $V_n\in\cB$ is the last open set in the chain $c$ determining $W=W_{c,t}$ and the vertical maps are induced by the maps $p$ and $q$ in \eqref{ePullbackSquare}. The first thing to check is that both vertical maps are homotopy equivalences. Note that $p|_W\colon W\to V_n$ is a homotopy equivalence (since it is a map between contractible spaces) and a Hurewicz fibration, so its pullback along any map is again a homotopy equivalence. The middle vertical map in \eqref{eComparingFibers} is its pullback along $r|_{V_n}\colon r^{-1}(V_n)\to V_n$.

For a space $X$ and subspace $A$ write $P_A X$ for the pathspace $\mathrm{Map}([0,1],\{0\};X,A)$, so for example $PZ=P_{\{z_0\}}Z$. Pulling back \eqref{ePullbackSquare} along the maps $P_W PZ\to PZ$ and $P_{V_n}Z\to Z$ (given by evaluation at $1$) gives another pullback square
\begin{equation}\label{ePullbackSquare2}
\centering
\begin{split}
\begin{tikzpicture}
[x=1mm,y=1mm]
\node (tl) at (0,12) {$\hofib_W(g)$};
\node (tr) at (30,12) {$P_W PZ$};
\node (bl) at (0,0) {$\hofib_{V_n}(r|_Z)$};
\node (br) at (30,0) {$P_{V_n}Z.$};
\draw[->] (tl) to (tr);
\draw[->] (bl) to (br);
\draw[->] (tl) to (bl);
\draw[->] (tr) to (br);
\node at (4,8) {$\lrcorner$};
\end{tikzpicture}
\end{split}
\end{equation}
The right-hand vertical map is a homotopy equivalence (since it is a map between contractible spaces) and a Hurewicz fibration, and therefore so is the left-hand vertical map, which is the right-hand vertical map of \eqref{eComparingFibers}.

Since $\hofib_W(g) \m \hofib_{V_n}(r|_Z)$ is a homotopy equivalence, there is a local coefficient system $\cG^\prime$ on $\hofib_{V_n}(r|_Z)$ which pulls back to $\cG$ on $\hofib_W(g)$. Moreover we may take $\cG^\prime$ to be the pullback of $\cG$ along a homotopy inverse, so $\cG^\prime$ is again in $\CC$. Since $V_n$ is part of a basis witnessing that $r$ is a Leray $\CC$-homology fibration, $r^{-1}(\gamma(1)) \hookrightarrow r^{-1}(V_n)$ is a $\cG^\prime$-homology isomorphism. Hence $g^{-1}(\gamma) \m g^{-1}(W)$ is also a $\cG^\prime$-homology isomorphism, i.e.\ a $\cG$-homology isomorphism, as required. This completes the verification that $g$ is a Leray $\CC$-homology fibration witnessed by the basis $\cW$.

Since $Z$ and $r^{-1}(Z)$ were assumed to be Hausdorff, so are the spaces $F$ and $PZ$. Hence the map $g\colon F\to PZ$, together with the basis $\cW$ for $PZ$ and the local coefficient system $j^*\cF$ on $F$, satisfies the conditions of Lemma \ref{contractibleCase}. This tells us that the inclusion $g^{-1}(\alpha_0)\emb F$ is an isomorphism on homology with coefficients pulled back from $j^*\cF$. But earlier we identified this inclusion with the map $i\colon r^{-1}(z_0)\to \hofib_{z_0}(r|_Z)$. Hence $i$ is an $\cF$-homology isomorphism, as required.
\end{proof}

\section{A twisted homology fibration criterion}\label{ssCriterion}

In this section we generalize Proposition 5.1 of \cite{Mc1} to give a criterion for a map $r\colon Y\to X$ to be a Serre $\CC$-homology fibration (and hence by Proposition \ref{SimpL} also a Leray $\CC$-homology fibration if it is locally stalk-like on $\CC$-homology). The criterion is the following theorem.

\begin{theorem}\label{strongcrit}
Let $X$ be a topological space with closed filtration $\{X_n\}_{n\in\bN}$, meaning that the $X_n$ are closed subsets of $X$ satisfying $X_{n-1}\subseteq X_n$, $X=\bigcup_{n\in\bN} X_n$ and each compact subset of $X$ is contained in some $X_n$. Let $r\colon Y \to X$ be a map and assume that each $X_n$ and $r^{-1}(X_n)$ is Hausdorff. Then $r$ is a Serre $\CC$-homology fibration if the following three conditions are satisfied\textup{:}
\begin{itemize2}
\item[\textup{(i)}] For each $n$ there is $X_{n-1}\subseteq U_n\subseteq X_n$, with $U_n$ open in $X_n$, such that $r$ is locally stalk-like on $\CC$-homology over $U_n$.
\item[\textup{(ii)}] The map $r$ is a Leray $\CC$-homology fibration on each difference $X_n \smallsetminus X_{n-1}$ and a Serre $\CC$-homology fibration on $X_0$.
\item[\textup{(iii)}] There are homotopies $h_t \colon U_n \m U_n$ and $H_t \colon r^{-1}(U_n) \m r^{-1}(U_n)$ satisfying\textup{:}
  \begin{itemize2}
  \item[\textup{(a)}] $h_0=id$, $h_t(X_{n-1}) \subseteq X_{n-1}$, $h_1(U_n) \subseteq X_{n-1}$\textup{;}
  \item[\textup{(b)}] $H_0 = id$, $r \circ H_t = h_t \circ r$\textup{;}
  \item[\textup{(c)}] for all $x \in U_n$, $H_1\colon r^{-1}(x) \m r^{-1}(h_1(x))$ induces an isomorphism on homology with coefficients in $\CC$ coming from $\hofib_{h_1(x)}(r)$.
  \end{itemize2}
\end{itemize2}
\end{theorem}
\begin{remark}\label{rMinimumHypotheses}
These are the minimum necessary hypotheses to prove the theorem. A slightly stronger version of (i) is to assume that $r$ is locally stalk-like over each $X_n$. The property of being locally stalk-like passes to open subspaces, so this implies (i) and also means via Propositions \ref{SimpL} and \ref{LimpS} that in (ii) we need not distinguish between Leray and Serre $\CC$-homology fibrations.

In practice one normally checks conditions (i) and (ii) of Theorem \ref{strongcrit} by showing that $r$ is locally stalk-like over each $X_n$ and a fibration over each $X_n \smallsetminus X_{n-1}$.
\end{remark}
\begin{proof}
We will prove by induction that $r$ is a Serre $\CC$-homology fibration on $X_n$ for all $n$. This will prove the theorem by the following argument. For any $x\in X$ the map $r^{-1}(x)\to \hofib_x(r)$ factors as
\[
r^{-1}(x) \longrightarrow \colim_n \hofib_x(r|_{X_n}) \xrightarrow{\iota} \hofib_x(r)
\]
and we know that the first map is an isomorphism on $\CC$-homology by what we will prove below and the fact that homology with twisted coefficients commutes with colimits. It will therefore suffice to prove that $\iota$ is a weak equivalence.

Note that $\iota$ is clearly a continuous injection, and since each compact subset of $X$ is contained in some $X_n$ it is also a surjection. Another consequence of the fact that each compact subset of $X$ is contained in some $X_n$ is that each continuous map $K\to \hofib_x(r)$ from a compact space $K$ factors through the map $\hofib_x(r|_{X_n}) \to \hofib_x(r)$ for some $n$. This means that, although $\iota^{-1}$ may be discontinuous, the composition $\iota^{-1}\circ f$ for any continuous map $f\colon K\to \hofib_x(r)$ from a compact space $K$ \emph{is} continuous. So $\iota$ induces a bijection between the set of continuous maps from a compact space into $colim_n \hofib_x(r|_{X_n})$ and the set of continuous maps from a compact space into $\hofib_x(r)$. Hence $\iota$ is a weak equivalence.

It remains to prove that $r$ is a Serre $\CC$-homology fibration on $X_n$ for all $n$. For the base case $n=0$ this is assumed by (ii). So let $n\geq 0$ and assume by induction that $r$ is a Serre $\CC$-homology fibration on $X_n$; we will show that it is a Serre $\CC$-homology fibration on $X_{n+1}$. To do this, we will first prove that it is a Serre $\CC$-homology fibration on $U_{n+1}$. Fix $x \in U_{n+1}$ and consider the following commutative diagram:
\begin{center}
\begin{tikzpicture}
[x=1.4mm,y=1.2mm]
\node (tl) at (0,12) {$r^{-1}(x)$};
\node (tm) at (25,12) {$\hofib_{x}(r|_{U_{n+1}})$};
\node (tr) at (50,12) {$\hofib_x(r)$};
\node (bl) at (0,0) {$r^{-1}(h_1(x))$};
\node (bm) at (25,0) {$\hofib_{h_1(x)}(r|_{X_n})$};
\node (br) at (50,0) {$\hofib_{h_1(x)}(r).$};
\draw[->] (tl) to (tm);
\draw[->] (tm) to (tr);
\draw[->] (bl) to (bm);
\draw[->] (bm) to (br);
\draw[->] (tl) to node[left,font=\small]{$H_1$} (bl);
\draw[->] (tm) to node[right,font=\small]{$H_1^\prime$} (bm);
\draw[->] (tr) to node[right,font=\small]{$\simeq$} (br);
\end{tikzpicture}
\end{center}
Here $H_1^\prime \colon \hofib_{x}(r|_{U_{n+1}}) \m \hofib_{h_1(x)}(r|_{X_n})$ is the map induced on homotopy fibers by the maps $h_1$ and $H_1$. Note that conditions (iii)(a) and (iii)(b) imply that $H_1^\prime$ is a homotopy equivalence. Now fix a local coefficient system $\cF \to \hofib_x(r)$ in $\CC$; this pulls back from a local coefficient system $\F^\prime \to \hofib_{h_1(x)}(r)$ since the right-hand vertical map above is a homotopy equivalence. We may take $\cF^\prime$ to be the pullback of $\cF$ along a homotopy inverse, so $\cF^\prime$ is again in $\CC$. The map $H_1$ induces an isomorphism on homology with coefficients pulled back from $\cF^\prime$ by condition (iii)(c). The inclusion $r^{-1}(h_1(x)) \m \hofib_{h_1(x)}(r|_{X_n})$ induces an isomorphism on homology with coefficients pulled back from $\cF^\prime$ since $r$ is a Serre $\CC$-homology fibration on $X_n$ by our inductive hypothesis. Thus the inclusion $r^{-1}(x) \m  \hofib_{x}(r|_{U_{n+1}})$ induces an isomorphism on homology with coefficients pulled back from $\cF$, and so $r$ is a Serre $\CC$-homology fibration on $U_{n+1}$.

By (i) and Proposition \ref{SimpL}, $r$ is a Leray $\CC$-homology fibration on $U_{n+1}$. We have also assumed in (ii) that $r$ is a Leray $\CC$-homology fibration on $X_{n+1} \smallsetminus X_n$. It is clear that if a subspace $Z$ is the union of two open (in $Z$) subsets $V_1$ and $V_2$ then $r$ is a Leray $\CC$-homology fibration on $Z$ if it is on $V_1$ and on $V_2$. Since $X_{n+1} = (X_{n+1}\smallsetminus X_n) \cup U_{n+1}$ and $r$ is a Leray $\CC$-homology fibration on $X_{n+1}\smallsetminus X_n$ and $U_{n+1}$, the map $r$ is also a Leray $\CC$-homology fibration on $X_{n+1}$. Hence by Proposition \ref{LimpS} it is a Serre $\CC$-homology fibration on $X_{n+1}$, finishing the inductive step.
\end{proof}

\begin{remark}
McDuff's homology fibration criterion in \cite{Mc1} has some additional assumptions about spaces having nice local properties and having the homotopy type of CW complexes. These assumptions stem from her use of \v{C}ech cohomology and the need to equate \v{C}ech and singular cohomology. Since we use semi-simplicial spaces instead of sheaves to construct the spectral sequences used in proving Lemma \ref{contractibleCase} (which feeds into the proof of Theorem \ref{strongcrit}), these assumptions turn out not to be needed. This is similar to the fact that the paracompactness assumptions of Propositions 5 and 6 in \cite{MSe}, which are used to prove their group-completion theorem, are unnecessary, as pointed out in \cite[pp.\ 56--57]{Se} and \cite[pp.\ 109--110]{Mc3}.
\end{remark}

\section{The twisted group-completion theorem}\label{ss-group-completion}

In this section we prove the twisted (or $\CC$-) version of McDuff-Segal's group-completion theorem \cite[Proposition 2]{MSe}. We do not claim any originality here as the proofs closely follow the ideas of \cite{MSe}. We also discuss results of Randal-Williams \cite{Randal-Williams2013Groupcompletion} regarding applying the group-completion theorem to homotopy-commutative monoids. 

\paragraph{Geometric realizations.}
The geometric realization $\dnorm{Y_\bullet}$ of a semi-simplicial space $Y_\bullet$ can be defined inductively as follows: first $\dnorm{Y_\bullet}^{0} = Y_0$, then $\dnorm{Y_\bullet}^{n}$ is the pushout of the diagram
\begin{equation}\label{ePushoutDiagram}
\dnorm{Y_\bullet}^{n-1} \leftarrow \partial\Delta^n \times Y_n \to \Delta^n \times Y_n
\end{equation}
for each $n\geq 1$ and then $\dnorm{Y_\bullet}$ is the colimit of $\dnorm{Y_\bullet}^{n}$ as $n\to\infty$. The left-hand map in \eqref{ePushoutDiagram} is determined by the face maps of $Y_\bullet$. The right-hand map in \eqref{ePushoutDiagram} is a cofibration, so $\dnorm{Y_\bullet}^n$ is a model for the homotopy pushout of \eqref{ePushoutDiagram}. In fact, $\dnorm{Y_\bullet}^n$ is homeomorphic to the double mapping cylinder of \eqref{ePushoutDiagram}. The inclusions $\dnorm{Y_\bullet}^n \emb \dnorm{Y_\bullet}^{n+1}$ are NDR-pairs, hence cofibrations, so $\dnorm{Y_\bullet}$ is a model for the homotopy colimit of $\dnorm{Y_\bullet}^{n}$ as $n\to\infty$.

\paragraph{The $\CC$-group-completion theorem.}
Let $\cM$ be a topological monoid acting on a space $X$. Associated to this there is a map $p\colon E\cM \times_{\cM} X \to B\cM$ from the Borel construction to the classifying space of the monoid. Explicitly we take the following point-set model for this map. Let $E_n = \cM^n \times X$ and $B_n = \cM^n$ with $p_n\colon E_n \to B_n$ the projection onto the first factor. This is a map of semi-simplicial spaces, where $E_\bullet$ and $B_\bullet$ are given the usual face maps from the bar construction. We then take $p\colon E\cM \times_{\cM} X \to B\cM$ to be the geometric realization of this semi-simplicial map.

\begin{theorem}\label{pGroupCompletion}
Let $\cM$ be a Hausdorff, locally contractible monoid acting on a Hausdorff space $X$. Suppose that for all $m\in \cM$ the action map $m\cdot -\colon X\to X$ is a $\CC$-homology equivalence. Then $p$ is a Serre $\CC$-homology fibration.
\end{theorem}

Compared to Proposition 2 of \cite{MSe}, the hypothesis and conclusion have both been strengthened to all local coefficient systems in a fixed collection $\CC$ which is closed under pullbacks. In particular, when $\CC$ is the collection of all trivial coefficient systems this is precisely Proposition 2 of \cite{MSe}.

\paragraph{Homotopy-commutative monoids.}
We now discuss an application of the group-completion theorem to the homology of topological monoids. Let $\cM$ be a topological monoid with $\pi_0(\cM)=\bN$. (One could work in more generality, but we restrict to this case for ease of notation and because it is the only case that will be necessary in \cite{MP}.) Denote its components by $\cM_k$ and choose an element $m\in \cM_1$. We then form $\cM_\infty$ as the mapping telescope of the sequence $\cM\to \cM\to \cM\to\cdots$ where each map is right-multiplication by $m$. There is then an induced left-action of $\cM$ on $\cM_\infty$.

If we now assume that $\cM$ is homotopy-commutative, then for each $m^\prime\in \cM$ this action $m^\prime \cdot -\colon \cM_\infty \to \cM_\infty$ is a trivial homology equivalence (i.e.\ with trivial $\bZ$ coefficients). To see this: say $m^\prime \in \cM_k$. Then the the map we are interested in is the map induced on $\cM_\infty$ by the vertical maps in the diagram
\begin{equation}\label{eTriangles}
\centering
\begin{split}
\begin{tikzpicture}
[x=1mm,y=1.5mm]
\node (t1) at (0,10) {$\cdots$};
\node (t2) at (20,10) {$\cM$};
\node (t3) at (60,10) {$\cM$};
\node (t4) at (80,10) {$\cdots$};
\node (b1) at (0,0) {$\cdots$};
\node (b2) at (20,0) {$\cM$};
\node (b3) at (60,0) {$\cM$};
\node (b4) at (80,0) {$\cdots$};
\draw[->] (t1) to (t2);
\draw[->] (t2) to node[above,font=\small]{$-\cdot m^k$} (t3);
\draw[->] (t3) to (t4);
\draw[->] (b1) to (b2);
\draw[->] (b2) to node[below,font=\small]{$-\cdot m^k$} (b3);
\draw[->] (b3) to (b4);
\draw[->] (t2) to node[left,font=\small]{$m^\prime \cdot -$} (b2);
\draw[->] (t3) to node[right,font=\small]{$m^\prime \cdot -$} (b3);
\draw[->] (b2) to node[anchor=south,font=\small]{$\mathrm{id}$} (t3);
\end{tikzpicture}
\end{split}
\end{equation}
in which the triangles commute up to homotopy. This splitting into triangles induces a factorization on homology which implies that the induced map on the mapping telescope $\cM_\infty$ is a homology equivalence.

Note that $E\cM \times_{\cM} \cM_\infty = \dnorm{\cM^\bullet \times \cM_\infty}$ is the homotopy colimit of infinitely many copies of $\dnorm{\cM^\bullet \times \cM}$. The latter space is a model for $E\cM$, so it is weakly contractible, and therefore so is $E\cM \times_{\cM} \cM_\infty$. Hence the homotopy fiber of the map $p$ over any point is weakly equivalent to $\Omega B\cM$. Applying the group-completion theorem \cite[Proposition 2]{MSe} we therefore obtain a homology equivalence
\begin{equation}\label{eGroupCompletion}
\cM_\infty \cong \mathit{fib}(p) \longrightarrow \hofib(p) \simeq_{\mathsf{w}} \Omega B\cM
\end{equation}
(this is essentially Proposition 1 of \cite{MSe}). We would like to know that this map is in fact a \emph{twisted} homology equivalence, in other words an acyclic map.

Similarly to the discussion above one can show that, for homotopy-commutative monoids $\cM$, the maps $m^\prime \cdot -\colon \cM_\infty \to \cM_\infty$ are surjective on homology with coefficients in any local coefficient system on $\cM_\infty$. But the argument fails for injectivity,\footnote{To apply the argument, one needs to produce a factorization into triangles as in \eqref{eTriangles}, but for spaces equipped with local coefficient systems, thought of as bundles of abelian groups. For the bottom-right triangle (corrsponding to surjectivity) this can be done since given homotopic maps $f,g\colon X\rightrightarrows Y$ and a bundle $\cF$ over $Y$, one can factor the pullback along $f$ as a bundle map covering $\mathrm{id}_X$ followed by the pullback along $g$. However, one cannot in general factor it as the pullback along $g$ followed by a bundle map covering $\mathrm{id}_Y$ -- this is the problem one needs to solve in the top-left triangle, corresponding to injectivity.} and indeed it is in general \emph{not} true that homotopy-commutative monoids $\cM$ act on their mapping telescopes $\cM_\infty$ by acyclic maps; see Remark 2.6 of \cite{Randal-Williams2013Groupcompletion} for an example. However, it has been proved in \cite[\S 2.4]{Randal-Williams2013Groupcompletion} that the maps $m^\prime \cdot -\colon \cM_\infty \to \cM_\infty$ are nevertheless injective on homology with coefficients in any \emph{abelian} local coefficient system on $\cM_\infty$. Hence we may apply the $\CC$-group-completion theorem (Theorem \ref{pGroupCompletion}), with $\CC$ the subcategory $\mathfrak{Ab}$ of abelian local coefficient systems, to obtain:
\begin{corollary}\label{cGroupCompletion}
For a homotopy-commutative monoid $\cM$ the map \eqref{eGroupCompletion} is acyclic.
\end{corollary}
\begin{proof}
First, let $\snorm{S_\bullet(\cM)}$ be the thin geometric realization of the singular simplicial set of $\cM$. The functor $S_\bullet$ preserves all limits (since it is a right adjoint), and thin geometric realization preserves finite limits, so $\snorm{S_\bullet(-)}$ preserves finite products and therefore $\snorm{S_\bullet(\cM)}$ is again a homotopy-commutative monoid, which is now additionally Hausdorff and locally contractible. There is a natural map $\snorm{S_\bullet(\cM)}\to \cM$ which is a weak equivalence, so acyclicity of the map \eqref{eGroupCompletion} for $\cM$ replaced by $\snorm{S_\bullet(\cM)}$ will imply acyclicity of the map \eqref{eGroupCompletion} itself. So we may assume that $\cM$ is Hausdorff and locally contractible.

By the discussion above the action of $\cM$ on $\cM_\infty$ satisfies the hypothesis of Theorem \ref{pGroupCompletion} for $\CC=\mathfrak{Ab}$, so $p$ is an abelian Serre homology fibration. Hence (taking fibers and homotopy fibers over any point in $B\cM$) the map \eqref{eGroupCompletion} is an abelian homology equivalence. But all local coefficient systems on $\Omega B\cM$ are abelian, so it is a twisted homology equivalence, i.e.\ an acyclic map.
\end{proof}
By the classification of acyclic maps out of a given space (see for example \cite[Theorem IX.2.3]{AdemMilgram2004}) this means that $\Omega B\cM$ is weakly equivalent to the Quillen plus-construction of $\cM_\infty$ with respect to some perfect normal subgroup of $\pi_1(\cM_\infty)$. Since $\Omega B\cM$ has abelian $\pi_1$, this subgroup must be the commutator $[\pi_1(\cM_\infty),\pi_1(\cM_\infty)]$, so in particular we deduce that this commutator is perfect (this fact is proved directly in Proposition 3.1 of \cite{Randal-Williams2013Groupcompletion}). We also deduce that $\cM_{\infty}^+$ (where $(-)^+$ denotes plus-construction with respect to the maximal perfect subgroup) is a \emph{simple} space:
\begin{corollary}\label{cGroupCompletion2}
For a homotopy-commutative monoid $\cM$ the space $\cM_{\infty}^+$ is simple.
\end{corollary}
It now remains to prove Theorem \ref{pGroupCompletion}. It will follow easily from the more general fact:
\begin{proposition}\label{pGroupCompletion2}
Let $p_\bullet\colon E_\bullet\to B_\bullet$ be a map of semi-simplicial spaces which are levelwise Hausdorff, i.e.\ each $E_n$ and $B_n$ is Hausdorff. Suppose that each level $p_n\colon E_n\to B_n$ is a Leray $\CC$-homology fibration and for each face map $d_j\colon B_n\to B_{n-1}$ and element $b\in B_n$, the map
\begin{equation}\label{eRestrictionOfFaceMap}
d_j|_{p_n^{-1}(b)}\colon p_n^{-1}(b) \longrightarrow p_{n-1}^{-1}(d_j(b))
\end{equation}
is a $\CC$-homology equivalence. Then the induced map $\dnorm{p_\bullet} \colon \dnorm{E_\bullet} \to \dnorm{B_\bullet}$ of geometric realizations is a Serre $\CC$-homology fibration.
\end{proposition}
In applying Theorem \ref{pGroupCompletion} to obtain Corollaries \ref{cGroupCompletion} and \ref{cGroupCompletion2} we were interested in the fact that $p\colon E\cM \times_{\cM} X \to B\cM$ is an abelian \emph{Serre} homology fibration. However the proof of the next lemma, which is needed to prove Proposition \ref{pGroupCompletion2}, depends on the \emph{Leray} version of this notion; this is one of the reasons why we are led to consider both kinds of abelian homology fibrations.
\begin{lemma}\label{lStrongLerayHF1}
Suppose that we have a diagram\textup{:}
\begin{equation}\label{eStrongLerayHF1}
\centering
\begin{split}
\begin{tikzpicture}
[x=1mm,y=1mm]
\node (tl) at (0,15) {$E_1$};
\node (tm) at (30,15) {$E_0$};
\node (tr) at (60,15) {$E_2$};
\node (bl) at (0,0) {$B_1$};
\node (bm) at (30,0) {$B_0$};
\node (br) at (60,0) {$B_2$};
\draw[->] (tl) to node[right,font=\small]{$p_1$} (bl);
\draw[->] (tm) to node[right,font=\small]{$p_0$} (bm);
\draw[->] (tr) to node[right,font=\small]{$p_2$} (br);
\draw[->] (tm) to node[above,font=\small]{$g_1$} (tl);
\draw[->] (bm) to node[below,font=\small]{$f_1$} (bl);
\draw[->] (tm) to node[above,font=\small]{$g_2$} (tr);
\draw[->] (bm) to node[below,font=\small]{$f_2$} (br);
\end{tikzpicture}
\end{split}
\end{equation}
in which $p_i$ is a Leray $\CC$-homology fibration for $i=0,1,2$, and for all $b\in B_0$ the restriction $p_0^{-1}(b)\to p_i^{-1}(f_i(b))$ of $g_i$ is a $\CC$-homology equivalence for $i=1,2$.

Then the map $p\colon E\to B$ induced by taking double mapping cylinders levelwise -- in other words $E = E_1 \cup_{g_1} (E_0 \times [0,1]) \cup_{g_2} E_2$ etc.\ -- is also a Leray $\CC$-homology fibration.
\end{lemma}
We will prove this lemma first, then use it to deduce Proposition \ref{pGroupCompletion2}, and then finally show that this implies the ``$\CC$-group-completion theorem'' (Theorem \ref{pGroupCompletion}) as a special case.

If all the spaces involved are Hausdorff, one can deduce Lemma \ref{lStrongLerayHF1} from the $\CC$-homology fibration criterion (Theorem \ref{strongcrit}); see Remark \ref{rDeduceFromCriterion} below. However, we believe it is illuminating to give a direct proof too. Since this does not use the $\CC$-homology fibration criterion, it does not depend on the equivalence between Serre and Leray $\CC$-homology fibrations from \S\ref{ss-strong-homology-fib}.
\begin{proof}[Proof of Lemma \ref{lStrongLerayHF1}]
Let $\cU_i$ be a basis for $B_i$. Then the following is a basis $\cU$ for the double mapping cylinder $B = B_1 \cup_{f_1} (B_0 \times [0,1]) \cup_{f_2} B_2$:
\begin{equation*}\label{eDoubleMappingCylinder}
\begin{array}{rl|l}
\text{(a)} & V\cup_{f_1}\bigl( \bigcup_\alpha U_\alpha \times [0,\varepsilon_\alpha) \bigr) & V\in\cU_1, \varepsilon_\alpha >0 \text{ and } U_\alpha\in\cU_0 : \bigcup_\alpha U_\alpha = f_1^{-1}(V)\\
\text{(b)} & U\times (\beta,\gamma)  & U\in\cU_0 \text{ and } 0<\beta <\gamma <1\\
\text{(c)} & \bigl( \bigcup_\alpha U_\alpha \times (1\!-\!\varepsilon_\alpha,1] \bigr)\cup_{f_2} V & V\in\cU_2, \varepsilon_\alpha >0 \text{ and } U_\alpha\in\cU_0 : \bigcup_\alpha U_\alpha = f_2^{-1}(V).
\end{array}
\end{equation*}
Pictorially:
\begin{center}
\begin{tikzpicture}
[x=1mm,y=1mm,font=\small]
\begin{scope}
\draw[densely dotted] (0,0)--(5,0)--(5,3)--(10,3)--(10,7)--(7,7)--(7,10)--(0,10);
\draw[very thick] (0,0)--(0,10);
\draw[decorate,decoration={brace,amplitude=4pt}] (-1,0)--(-1,10);
\node at (-2,5) [anchor=east] {$V$};
\draw[decorate,decoration={brace,amplitude=4pt,mirror}] (11,0)--(11,10);
\node at (12,5) [anchor=west] {$f_1^{-1}(V)$};
\node at (0,-3) [color=black!70,font=\footnotesize] {$0$};
\node at (5,-7) [font=\normalsize] {(a)};
\end{scope}
\begin{scope}[xshift=40mm]
\draw[densely dotted] (0,0) rectangle (10,10);
\draw[decorate,decoration={brace,amplitude=4pt,mirror}] (11,0)--(11,10);
\node at (12,5) [anchor=west] {$U$};
\node at (0,-3) [color=black!70,font=\footnotesize] {$\beta$};
\node at (10,-3) [color=black!70,font=\footnotesize] {$\gamma$};
\node at (5,-7) [font=\normalsize] {(b)};
\end{scope}
\begin{scope}[xshift=90mm]
\draw[densely dotted] (10,0)--(4,0)--(4,4)--(0,4)--(0,8)--(7,8)--(7,10)--(10,10);
\draw[very thick] (10,0)--(10,10);
\draw[decorate,decoration={brace,amplitude=4pt}] (-1,0)--(-1,10);
\node at (-2,5) [anchor=east] {$f_2^{-1}(V)$};
\draw[decorate,decoration={brace,amplitude=4pt,mirror}] (11,0)--(11,10);
\node at (12,5) [anchor=west] {$V$};
\node at (10,-3) [color=black!70,font=\footnotesize] {$1$};
\node at (5,-7) [font=\normalsize] {(c)};
\end{scope}
\end{tikzpicture}
\end{center}
It is not enough to simply take (\v{a}): sets of the form $V\cup_{f_1}\bigl( f_1^{-1}(V)\times [0,\varepsilon) \bigr)$ (and similarly (\v{c})), by the following counterexample pointed out to the authors by Ilya Grigoriev. Take $B_0 = \bR$ and $B_1 = B_2 = \point$. Then the subset $\point\cup\{(s,t) \,|\, s<(1+t^2)^{-1} \}$ of the double mapping cylinder is open but is not covered by sets of the form (\v{a}), (b) and (\v{c}).\footnote{In \cite{MSe} it appears to be assumed (in the proof of Proposition 3) that the collection (\v{a}), (b) and (\v{c}) is a basis for the double mapping cylinder. This is not a major problem though, since the proof is not essentially made any more complicated by having to admit sets of the more general form (a) and (c).} But if one allows the more general sets of the form (a) and (c) then it is not hard to check that this is, as claimed, a basis for the double mapping cylinder.

If the bases $\cU_i$ consist of contractible sets then so will $\cU$. Now assume that $\cU_i$ is a basis for $B_i$ witnessing that $p_i\colon E_i\to B_i$ is a Leray $\CC$-homology fibration. We will show that $p\colon E\to B$ is also a Leray $\CC$-homology fibration witnessed by the basis $\cU$ for $B$. To do this we need to show that for any $b\in W\in \cU$, the inclusions
\[
p^{-1}(b) \xrightarrow{i} p^{-1}(W) \xrightarrow{j} \hofib_W(p)\]
induce an isomorphism $H_*(p^{-1}(b);i^* j^* \cF) \cong H_*(p^{-1}(W);j^* \cF)$ for any local coefficient system $\cF$ on $\hofib_W(p)$ in $\CC$. There are three essentially different cases of $b\in W\in\cU$ to check:
\begin{center}
\begin{tikzpicture}
[x=1mm,y=1mm,font=\small]
\begin{scope}
\draw[densely dotted] (0,0)--(5,0)--(5,3)--(10,3)--(10,7)--(7,7)--(7,10)--(0,10);
\draw[very thick] (0,0)--(0,10);
\draw[decorate,decoration={brace,amplitude=4pt}] (-1,0)--(-1,10);
\node at (-2,5) [anchor=east] {$V$};
\node at (0,-4) [anchor=base,color=black!70,font=\footnotesize] {$0$};
\node at (0,5) [circle,fill,inner sep=1pt] {};
\node at (0,5) [anchor=west] {$b$};
\node at (5,-7) [font=\normalsize] {(i)};
\end{scope}
\begin{scope}[xshift=70mm]
\draw[densely dotted] (0,0) rectangle (10,10);
\draw[decorate,decoration={brace,amplitude=4pt,mirror}] (11,0)--(11,10);
\node at (12,5) [anchor=west] {$U$};
\node at (0,-4) [anchor=base,color=black!70,font=\footnotesize] {$\beta$};
\node at (3.5,-4) [anchor=base,color=black!70,font=\footnotesize] {$\delta$};
\node at (10,-4) [anchor=base,color=black!70,font=\footnotesize] {$\gamma$};
\node at (3,5) [circle,fill,inner sep=1pt] {};
\node at (3,5) [anchor=west] {$b$};
\node at (5,-7) [font=\normalsize] {(iii)};
\end{scope}
\begin{scope}[xshift=35mm]
\draw[densely dotted] (0,0)--(5,0)--(5,3)--(10,3)--(10,7)--(7,7)--(7,10)--(0,10);
\draw[very thick] (0,0)--(0,10);
\draw[decorate,decoration={brace,amplitude=4pt}] (-1,0)--(-1,10);
\node at (-2,5) [anchor=east] {$V$};
\node at (0,-4) [anchor=base,color=black!70,font=\footnotesize] {$0$};
\node at (3.5,-4) [anchor=base,color=black!70,font=\footnotesize] {$\delta$};
\node at (3,5) [circle,fill,inner sep=1pt] {};
\node at (3,5) [anchor=west] {$b$};
\node at (5,-7) [font=\normalsize] {(ii)};
\end{scope}
\end{tikzpicture}
\end{center}

Note that in case (i) the point $b$ is an element of $V\subseteq W$, whereas in cases (ii) and (iii) we have $b=(a,\delta)$ for some $a$ in $f_1^{-1}(V)$ or $U$, and $\delta >0$.

Suppose first we are in case (i) and fix a local coefficient system $\cF$ on $\hofib_W(p)$. We will say $\cF$-homology to mean homology with coefficients in pullbacks of $\cF$. The inclusion $p^{-1}(b) \hookrightarrow p^{-1}(W)$ factors through the inclusion $p^{-1}(V)\hookrightarrow p^{-1}(W)$, which is a homotopy equivalence since there is an evident deformation retraction of $p^{-1}(W)$ onto $p^{-1}(V)$. Hence it suffices to show that $p^{-1}(b)\hookrightarrow p^{-1}(V)$ induces an isomorphism on $\cF$-homology. But this is the same as $p_1^{-1}(b)\hookrightarrow p_1^{-1}(V)$, and the coefficients $\cF$ are pulled back to this through $\hofib_V(p_1)$, so this does induce an isomorphism on $\cF$-homology since $p_1$ is a Leray $\CC$-homology fibration (and $V$ is a part of a basis witnessing this).

Case (iii) is very similar to case (i), but case (ii) requires a little more care. Again fix a local coefficient system $\cF$ on $\hofib_W(p)$ that is in $\CC$. As before, $p^{-1}(W)$ deformation retracts onto $p^{-1}(V)$, so the inclusion $i\colon p^{-1}(b)\hookrightarrow p^{-1}(W)$ factors up to homotopy as
\[
p^{-1}(b) \to p^{-1}(f_1(a)) \hookrightarrow p^{-1}(V) \hookrightarrow p^{-1}(W),
\]
where the first map is a restriction of $g_1$ and the other two are inclusions. Call this composite map $i^\prime$.

The middle map is the same as $p_1^{-1}(f_1(a)) \hookrightarrow p_1^{-1}(V)$, and the coefficients $\cF$ are pulled back through $\hofib_V(p_1)$. Therefore it induces an isomorphism on $\cF$-homology since $p_1$ is a Leray $\CC$-homology fibration. The third map is a homotopy equivalence, and the first map is the same as $p_0^{-1}(a) \to p_1^{-1}(f_1(a))$, which is a $\CC$-homology equivalence by hypothesis. Hence the composite map $i^\prime$ induces an isomorphism on $\cF$-homology. Moreover, it is homotopic to the map $i$, so a choice of homotopy $i\simeq i^\prime$ induces an isomorphism making the triangle
\begin{center}
\begin{tikzpicture}
[x=1mm,y=1mm]
\node (tl) at (0,5) {$H_*(p^{-1}(b);i^* j^* \cF)$};
\node (bl) at (0,-5) {$H_*(p^{-1}(b);(i^\prime)^* j^* \cF)$};
\node (r) at (60,0) {$H_*(p^{-1}(W);j^* \cF)$};
\draw[->] (tl) to node[above,font=\small]{$i_*$} (r);
\draw[->] (bl) to node[below,font=\small]{$(i^\prime)_*$} (r);
\node at (0,0) {\rotatebox{270}{$\cong$}};
\end{tikzpicture}
\end{center}
commute. Hence $i$ also induces an isomorphism on $\cF$-homology, as required.
\end{proof}

\begin{remark}\label{rDeduceFromCriterion}
When the spaces $E_i$ and $B_i$ are all Hausdorff, one can alternatively deduce Lemma \ref{lStrongLerayHF1} from the $\CC$-homology fibration criterion (Theorem \ref{strongcrit}), as follows. Note that it is enough to show that $p$ is a Leray $\CC$-homology fibration on each of the mapping cylinders forming the double mapping cylinder. Filter the mapping cylinder $X=B_1\cup_{f_1}(B_0\times [0,1])$ by $X_0 = B_1$ and $X_n=X$ for $n\geq 1$, and take $U_1 = B_1\cup_{f_1}(B_0\times [0,\tfrac12))$ and $U_n=X$ for $n\geq 2$. Note that a mapping cylinder of Hausdorff spaces is Hausdorff.

Since $p_0$ and $p_1$ are Leray $\CC$-homology fibrations, they are locally stalk-like on $\CC$-homology. This property is inherited by the mapping cylinder,\footnote{To see this we need to take a basis for the mapping cylinder similar to the basis for the double mapping cylinder described in the proof of Lemma \ref{lStrongLerayHF1}.} so condition (i) of Theorem \ref{strongcrit} holds. Condition (ii) follows from the assumption that $p_0$ and $p_1$ are Leray $\CC$-homology fibrations, plus Proposition \ref{LimpS}. For (iii) we can use the deformation retraction which gradually squashes a mapping cylinder onto its base to define both $h_t$ and $H_t$; this satisfies (a) and (b) by construction. Finally, condition (iii)(c) follows from the assumption in Lemma \ref{lStrongLerayHF1} that $p_0^{-1}(b)\to p_1^{-1}(f_1(b))$ is a $\CC$-homology equivalence. Theorem \ref{strongcrit} therefore tells us that $p$ is a Serre $\CC$-homology fibration on the mapping cylinder $B_1\cup_{f_1}(B_0\times [0,1])$. But as mentioned above, $p$ is locally stalk-like on $\CC$-homology, so by Proposition \ref{SimpL} it is also a Leray $\CC$-homology fibration on the mapping cylinder.
\end{remark}
\begin{proof}[Proof of Proposition \ref{pGroupCompletion2}]
\hspace{0pt}\newline
\noindent $\bullet$ \emph{Step 1 $(n=0)$.} The map $\dnorm{p_\bullet}^0 \colon \dnorm{E_\bullet}^0 \to \dnorm{B_\bullet}^0$ of $0$-skeleta is just the map $p_0\colon E_0 \to B_0$, and so is a Leray $\CC$-homology fibration by hypothesis.

\noindent $\bullet$ \emph{Step 2 $(n\geq 1)$.} The map $\dnorm{p_\bullet}^n \colon \dnorm{E_\bullet}^n \to \dnorm{B_\bullet}^n$ of $n$-skeleta is the map of double mapping cylinders induced by
\begin{center}
\begin{tikzpicture}
[x=1mm,y=1mm]
\node (tl) at (0,15) {$\dnorm{E_\bullet}^{n-1}$};
\node (tm) at (30,15) {$\partial \Delta^n \times E_n$};
\node (tr) at (60,15) {$\Delta^n \times E_n$};
\node (bl) at (0,0) {$\dnorm{B_\bullet}^{n-1}$};
\node (bm) at (30,0) {$\partial \Delta^n \times B_n$};
\node (br) at (60,0) {$\Delta^n \times B_n.$};
\draw[->] (tm) to (tl);
\draw[->] (bm) to (bl);
\draw[->] (tl) to node[left,font=\small]{$\dnorm{p_\bullet}^{n-1}$} (bl);
\draw[->] (tm) to node[right,font=\small]{$1\times p_n$} (bm);
\draw[->] (tr) to node[right,font=\small]{$1\times p_n$} (br);
\incl{(tm)}{(tr)}
\incl{(bm)}{(br)}
\end{tikzpicture}
\end{center}
Hence we just need to verify the conditions of Lemma \ref{lStrongLerayHF1} in this case. The left vertical map is a Leray $\CC$-homology fibration by induction, and the other two vertical maps are too since $p_n$ is. The right-hand square induces homeomorphisms (and therefore $\CC$-homology equivalences) on set-theoretic fibers, since the horizontal maps are just inclusions. It therefore remains to prove that the left-hand square above induces $\CC$-homology equivalences on set-theoretic fibers. We can rewrite this square as:
\begin{center}
\begin{tikzpicture}
[x=1mm,y=1mm]
\node (tl) at (0,15) {$\bigl( \coprod_{k\leq n-1} (\Delta^k \times E_k) \bigr)/\sim$};
\node (tm) at (60,15) {$\partial \Delta^n \times E_n$};
\node (bl) at (0,0) {$\bigl( \coprod_{k\leq n-1} (\Delta^k \times B_k) \bigr)/\sim$};
\node (bm) at (60,0) {$\partial \Delta^n \times B_n$};
\draw[->] (tm) to node[above,font=\small]{$g_n$} (tl);
\draw[->] (bm) to node[below,font=\small]{$f_n$} (bl);
\draw[->] (tl) to node[left,font=\small]{$\dnorm{p_\bullet}^{n-1}$} (bl);
\draw[->] (tm) to node[right,font=\small]{$1\times p_n$} (bm);
\end{tikzpicture}
\end{center}
Let $(a,b)\in \partial\Delta^n \times B_n$. The fiber of $1\times p_n$ over this point is $\{a\} \times p_n^{-1}(b)$. Its image $f_n(a,b)$ may have multiple representatives, but it has a unique representative $(a^\prime,b^\prime)\in \Delta^k \times B_k$ with $a^\prime$ in the interior of $\Delta^k$. The fiber of $\dnorm{p_\bullet}^{n-1}$ over this point is $\{a^\prime\} \times p_k^{-1}(b^\prime)$. Note that $b^\prime = d_{i_{n-k}}\circ \dotsb\circ d_{i_1}(b)$ for some string of face operators. Write $b_j = d_{i_j}\circ \dotsb\circ d_{i_1}(b)$. The map of set-theoretic fibers may be identified with the composition
\[
p_n^{-1}(b) \longrightarrow p_{n-1}^{-1}(b_1) \longrightarrow p_{n-2}^{-1}(b_2) \longrightarrow \dotsb \longrightarrow p_k^{-1}(b_{n-k}) = p_k^{-1}(b^\prime)
\]
in which the $j$th map is the restriction of the face map $d_{i_j}\colon B_{n-j+1} \to B_{n-j}$. But these are all $\CC$-homology equivalences by hypothesis.

\noindent $\bullet$ \emph{Step 3 $(n=\infty)$.}
To finish the proof we will apply the $\CC$-homology fibration criterion. The geometric realization $X=\dnorm{B_\bullet}$ is filtered by its skeleta $X_n=\dnorm{B_\bullet}^n$, so we need to check the conditions of Theorem \ref{strongcrit} in this case. First note that each $B_n$ is Hausdorff by assumption, so each $\dnorm{B_\bullet}^n$ is also Hausdorff since it is formed by taking iterated double mapping cylinders as described in the previous step, and one can see that a double mapping cylinder of Hausdorff spaces is Hausdorff using the basis for the topology of a double mapping cylinder described in the proof of Lemma \ref{lStrongLerayHF1}.\footnote{In general, the geometric realization $\dnorm{B_\bullet}$ of a levelwise Hausdorff semi-simplicial space need not be Hausdorff (c.f.\ Proposition A.3 of \cite{Ebert2009}), but we only require each $n$-skeleton to be Hausdorff.} For the same reason the skeleta $\dnorm{E_\bullet}^n = (\dnorm{p_\bullet})^{-1}(\dnorm{B_\bullet}^n)$ are also Hausdorff. We now check conditions (i)--(iii).

Let $\dot{\Delta}^n$ denote the simplex $\Delta^n$ with its barycenter removed. Under the quotient maps
\[
\textstyle{\bigsqcup}_{n\geq 0}(\Delta^n \times B_n) \to \dnorm{B_\bullet} = X \qquad\text{and}\qquad \textstyle{\bigsqcup}_{n\geq 0}(\Delta^n \times E_n) \to \dnorm{E_\bullet} = Y,
\]
$X_n$ is the image of $\Delta^n \times B_n$ and $r^{-1}(X_n)$ is the image of $\Delta^n \times E_n$. We define $U_n$ to be the image of $\dot{\Delta}^n \times B_n$; note that $r^{-1}(U_n)$ is the image of $\dot{\Delta}^n \times E_n$. Pick any deformation retraction of $\dot{\Delta}^n$ onto $\partial\Delta^n$. This induces the homotopies $h_t$ and $H_t$ for condition (iii) of Theorem \ref{strongcrit}. Condition (iii)(c) follows from the assumption that the maps \eqref{eRestrictionOfFaceMap} are $\CC$-homology equivalences. By the previous step, $\dnorm{p_\bullet}$ is a Leray $\CC$-homology fibration on each skeleton $X_n$, so in particular it is locally stalk-like on $\CC$-homology. This property passes to the open subspaces $U_n$ of the $X_n$, so condition (i) is satisfied. For condition (ii): $\dnorm{p_\bullet}$ is a Leray $\CC$-homology fibration on $X_0$ by Step 1 above, and therefore by Proposition \ref{LimpS} a Serre $\CC$-homology fibration. The difference $X_n \smallsetminus X_{n-1}$ is $\mathrm{int}(\Delta^n) \times B_n$, over which $\dnorm{p_\bullet}$ is simply $1\times p_n$, which is a Leray $\CC$-homology fibration since $p_n$ is assumed to be one. Thus Theorem \ref{strongcrit} implies that $\dnorm{p_\bullet}$ is a Serre $\CC$-homology fibration.
\end{proof}
\begin{proof}[Proof of Theorem \ref{pGroupCompletion}]
We just need to check the conditions of Proposition \ref{pGroupCompletion2} for the map of semi-simplicial spaces $p_\bullet \colon E_\bullet \to B_\bullet$ where $E_n = \cM^n \times X$ and $B_n = \cM^n$, with the usual face maps of the bar construction, and $p_n \colon \cM^n \times X \to \cM^n$ is the projection. Since $\cM$ and $X$ are Hausdorff so is each $E_n$ and $B_n$.\footnote{In fact $\dnorm{E_\bullet}$ and $\dnorm{B_\bullet}$ are Hausdorff. Although levelwise Hausdorff semi-simplicial spaces do not in general have Hausdorff geometric realizations, it is true for nerves of Hausdorff monoids. One can picture points of $\dnorm{B_\bullet} = \dnorm{\cM^\bullet}$ as configurations on $(0,1)$ labelled by elements of $\cM$ which multiply when points collide and where points may fall off the ends of the interval. Points of $\dnorm{E_\bullet} = \dnorm{\cM^\bullet \times X}$ may be viewed as similar configurations on $(0,1]$, where in addition the point $1$ is labelled by an element of $X$ which is acted on when a point labelled in $\cM$ collides with it. In this picture it is easy to write down separating open neighborhoods for any pair of distinct points.} Also, $\cM^n$ is locally contractible, so the trivial fiber bundle $p_n$ over it is a Leray $\CC$-homology fibration. We also need to check that for all face maps
\begin{center}
\begin{tikzpicture}
[x=1mm,y=1mm]
\node (tl) at (0,15) {$\cM^n \times X$};
\node (tr) at (30,15) {$\cM^n$};
\node (bl) at (0,0) {$\cM^{n-1} \times X$};
\node (br) at (30,0) {$\cM^{n-1}$};
\draw[->] (tl) to node[above,font=\small]{$p_n$} (tr);
\draw[->] (bl) to node[below,font=\small]{$p_{n-1}$} (br);
\draw[->] (tl) to node[left,font=\small]{$d_j$} (bl);
\draw[->] (tr) to node[right,font=\small]{$d_j$} (br);
\end{tikzpicture}
\end{center}
and elements $b=(m_1,\ldots,m_n)\in \cM^n$, the map $p_n^{-1}(b) \to p_{n-1}^{-1}(d_j(b))$ is a $\CC$-homology equivalence. For $0\leq j<n$ this map is just the identity $X\to X$. For $j=n$, it is the map $m_n\cdot -\colon X\to X$ which acts on $X$ by $m_n$. But this is a $\CC$-homology equivalence by hypothesis.
\end{proof}

\small
\phantomsection
\addcontentsline{toc}{section}{References}
\renewcommand{\bibfont}{\normalfont\small}
\defbibnote{myprenote}{Links to ArXiv versions of published articles are provided in \{braces\}.}
\printbibliography[prenote=myprenote]

\normalsize

\vspace*{1em}

\noindent{\small\scshape Department of Mathematics, Stanford University, Building 380, Office 383A, Stanford, California, 94305.}

\noindent{\sffamily jkmiller@math.stanford.edu}

\vspace*{1em}

\noindent {\small\scshape Mathematisches Institut, WWU M{\"u}nster, Einsteinstra{\ss}e 62, 48149 M{\"u}nster, Germany}

\noindent {\sffamily mpalm{\textunderscore}01@uni-muenster.de}

\end{document}